\documentclass[a4paper,UKenglish,cleveref, autoref, thm-restate]{lipics-v2021}

\newcommand{\NP}{{\sf NP}}

\title{On Detecting $H$-Induced Minors for Small $H$}
\titlerunning{On Detecting $H$-Induced Minors for Small $H$}

\author{Tala Eagling-Vose}{Department of Computer Science, Durham University, Durham, UK}{tala.j.eagling-vose@durham.ac.uk}{https://orcid.org/0009-0008-0346-7032}{}
\author{Barnaby Martin}{Department of Computer Science, Durham University, Durham, UK}{barnaby.d.martin@durham.ac.uk}{https://orcid.org/0000-0002-4642-8614}{supported by the Leverhulme Trust (RPG-2024-182).}
\author{Dani\"el Paulusma}{Department of Computer Science, Durham University, Durham, UK}{daniel.paulusma@durham.ac.uk}{https://orcid.org/0000-0001-5945-9287}{supported by the Leverhulme Trust (RPG-2024-182).}
\author{Nicolas Trotignon}{ENS de Lyon, CNRS, Lyon, France}{nicolas.trotignon@ens-lyon.fr}{https://orcid.org/0000-0003-1978-0687}{}

\authorrunning{T. Eagling-Vose, B. Martin, D. Paulusma and N. Trotignon}

\Copyright{Tala Eagling-Vose, Barnaby Martin, Dani\"el Paulusma and Nicolas Trotignon}



\category{}
\relatedversion{}
\acknowledgements{This work has benefitted from Dagstuhl Seminars 22481 and 25041, in which the second, third and fourth author participated. In particular, one of the results resolves the open problem asked in~\cite{CMPSA23}. The authors also thank Jungho Ahn for initial discussions on the problem.}

\nolinenumbers
\hideLIPIcs  

\EventEditors{John Q. Open and Joan R. Access}
\EventNoEds{2}
\EventLongTitle{42nd Conference on Very Important Topics (CVIT 2016)}
\EventShortTitle{CVIT 2016}
\EventAcronym{CVIT}
\EventYear{2016}
\EventDate{December 24--27, 2016}
\EventLocation{Little Whinging, United Kingdom}
\EventLogo{}
\SeriesVolume{42}
\ArticleNo{23}

\begin{document}
\maketitle

\begin{abstract}
We consider the {\sc $H$-Induced Minor} problem: for a fixed graph~$H$, decide whether a given graph $G$ contains $H$ as an induced minor. While the problem is known to be \NP-complete for some trees~$H$ on more than $2^{300}$ vertices, the complexity for small trees remains unresolved. In particular, the case where $H$ is the $7$-vertex tree consisting of a path on five vertices with a pendant vertex attached to the second and fourth vertex was a long-standing open problem. We show that this case is polynomial-time solvable by developing algorithms that detect a sequence of carefully chosen substructures. Complementing this, we prove that  detecting some of these substructures individually is \NP-hard. We also give polynomial-time algorithms for three cases where $H$ is a graph on five vertices (that is not a tree). In this way, we completed the classification of $H$-{\sc Induced Minor} for graphs $H$ on five vertices and answered an open problem of Dallard, Dumas, Hilaire and Perez~(2025).
\end{abstract}

\section{Introduction}\label{s-intro}

For a fixed graph~$H$, we study the {\sc $H$-Induced Minor} problem. This problem asks whether a given graph~$G$ contains $H$ as an {\it induced minor}; that is, whether $G$ can be transformed into~$H$ through a sequence of vertex deletions and edge contractions. The problem thus belongs to the broader class of graph modification problems, a central topic in both algorithmic and structural graph theory. In this setting, the goal is to determine whether a graph $G$ can be transformed into a graph from a specified family~${\cal H}$ using a sequence of operations from a fixed set~$S$. For example, if ${\cal H}=\{H\}$, but $S$ consists of vertex deletion, edge contraction {\it and} edge deletion, then we obtain the classical {\sc $H$-Minor} problem. If ${\cal H}=\{H\}$ and $S$ consists of edge contraction, then we obtain another well-known problem, namely $H$-{\sc Contraction}.

It is well known that {\sc $H$-Minor} can be solved in cubic time for every graph~$H$~\cite{RS95} (see~\cite{KPS24} for an almost-linear time algorithm). However, the problem $H$-{\sc Contraction} is far from being classified and is known to be \NP-complete even for small graphs $H$, such as the $4$-vertex path $P_4$~\cite{BV87} (see~\cite{LPW08,LPW08b,HKPST12} for further partial results).
In contrast to $H$-{\sc Minor} and similar to $H$-{\sc Contraction},
Fellows, Kratochv\'il, Middendorf and Pfeiffer~\cite{FKMP95} showed that {\sc $H$-Induced Minor} can also be \NP-complete. As an example, they provided a non-planar graph~$H$ on $68$ vertices. This is still the smallest known graph for which the problem is \NP-complete.
They also asked whether there exists a {\it planar} graph~$H$ for which {\sc $H$-Induced Minor} is \NP-complete. This question led to a series of follow-up results, which we describe below, including several recent developments.

\medskip
\noindent
{\bf Known Results.}
Even classifying the complexity of {\sc $H$-Induced Minor} is challenging. The problem was shown to be polynomial-time solvable for all trees~$H$ with at most seven vertices except when $H$ is the graph~$\mathbb{H}_2$ shown in Figure~\ref{fig:h2}~\cite{FKP12}. It is also polynomial-time solvable for all subdivided stars~$H$ and for all double stars~$H$, in which one of the two centre vertices has at most two leaves~\cite{FKP12}.
 
The polynomial-time results for star-like trees from~\cite{FKP12} were recently extended to generalized bulls, generalized houses,
complete split graphs with clique number at most~$4$, 
and to flowers by Dallard, Dumas, Hilaire and Perez~\cite{DDHP25} (a flower is an intersection of paths, cycles and diamonds in one vertex; for instance, a subdivided star is a flower). Moreover,
Bousquet et al.~\cite{BDDHMPT26} recently proved polynomial-time solvability of {\sc $H$-Induced Minor} if $H$ is the $4$-vertex wheel, the complete graph on five vertices minus an edge, or the complete bipartite graph~$K_{2,\ell}$ for all $\ell\geq 1$.

By using the above results and resolving remaining cases, Dallard et al.~\cite{DDHP25} proved polynomial-time solvability for all graphs~$H$ on at most five vertices except for the three open cases
displayed in Figure~\ref{fig:h2}.
In addition, Nguyen, Scott and Seymour~\cite{NSS24} gave a polynomial-time algorithm for the case $H=tC_3$ (disjoint union of $t$ triangles) for all $t\geq 2$; we note that the case $t=2$ was a long-standing open problem; see e.g.~\cite{FKP12,FKP13}.

In contrast to all these new polynomial-time results, Korhonen and Lokshtanov~\cite{KL24} recently answered the open problem of~\cite{FKMP95} by proving that there exists even a tree~$H$ (with more than $2^{300}$ vertices) for which {\sc $H$-Induced Minor} is \NP-complete.
Afterwards, Aboulker, Bonnet, Picavet and Trotignon~\cite{ABPT25} solved an open problem of Korhonen and Lokshtanov~\cite{KL24} by  
proving \NP-completeness for some specific subcubic graph~$H$ on $74$ vertices, in which every edge is incident to a vertex of degree~$2$.

From the above survey, it is readily seen that despite all the recent results, there are still many infinite families of graphs~$H$ for which the complexity of {\sc $H$-Induced Minor} is unknown. Therefore, the problem has also been studied on special graph classes. In particular, it is known that for every fixed graph $H$, {\sc $H$-Induced Minor} is polynomial-time solvable for AT-free graphs~\cite{GKP13}, chordal graphs~\cite{BGHHKP14}, $P_t$-free graphs, for all fixed $t\geq 1$~\cite{DDHP25}, planar graphs~\cite{FKMP95} or, more generally, all minor-closed graph classes that do not contain all graphs~\cite{HKPST12}; see~\cite{FKP13} for additional polynomial-time results for claw-free graphs.

\begin{figure}[t]
    \centering    \includegraphics[width=0.6\linewidth,page=1]{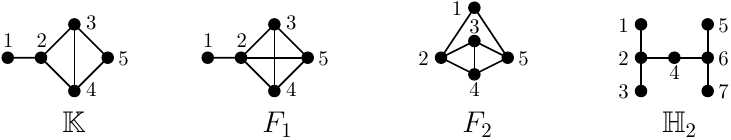}
    \caption{From left to right: the kite $\mathbb{K}$, the graphs $F_1$ and $F_2$, and the graph $\mathbb{H}_2$.}
    \label{fig:h2}
\end{figure}

\medskip
\noindent
{\bf Our Results.} 
We first prove in Sections~\ref{s-diamondplus}--\ref{s-f2} that $H$-{\sc Induced Minor} is polynomial-time solvable for 
every $H\in \{\mathbb{K},F_1,F_2\}$, which are the first three graphs displayed in Figure~\ref{fig:h2}. 
As a common technique, we first analyse the structure of a minimum induced minor model in each of these three cases, and we use our insights to reduce the problem to a polynomial number of instances of the {\sc $k$-Disjoint Subgraphs} problem, one of the main pillars in Graph Minor theory~\cite{RS95}.
Our results answer the open problem of Dallard et al.~\cite{DDHP25} and combined with the aforementioned existing results (see also~\cite{DDHP25}) leads to the following classification:

\begin{theorem}\label{t-dich5}
For every graph~$H$ on at most five vertices, {\sc $H$-Induced Minor} is polynomial-time solvable.
\end{theorem}

\noindent
In Section~\ref{s-h2}, we resolve the only open case of a tree~$H$ on seven vertices (see also~\cite{CMPSA23}). Namely, we prove
that $\mathbb{H}_2$-{\sc Induced Minor} is polynomial-time solvable for the tree $\mathbb{H}_2$ displayed in Figure~\ref{fig:h2} (to explain its name, $\mathbb{H}_2$ is obtained from the graph that looks like the letter ``H'' after subdividing the middle bar once). In addition, we observe that if {\sc $H$-Induced Minor} is polynomial-time solvable, then so is 
{\sc $(H+P_t)$-Induced Minor}. Namely, a graph $G$ has $H+P_t$ as an induced minor if and only if $G-N[S]$ has $H$ as an induced minor for some set~$S$ that induces a $P_t$ in $G$. Since $t$ is a fixed constant, we can branch on all possible options for choosing $S$.
Now, if $H$ is a disconnected forest on at most seven vertices, then one of its connected components is a path. Hence, our results combined with the aforementioned results from~\cite{FKP12} imply the following:

\begin{theorem}\label{t-dich7}
For every forest~$H$ on at most seven vertices, {\sc $H$-Induced Minor} is polynomial-time solvable.
\end{theorem}

\noindent
To detect $\mathbb{H}_2$ as an induced minor, we first detect specific models of it.  If none of these models are present, we show that any model must look like what we call a windmill.  We then rely on a \emph{shortest path detector}, a technique first used in \cite{chudnovsky.c.l.s.v:reco}, to detect a windmill (under the assumption that the specific models of the first step are not present). 
To show that a more naive approach is likely to fail, we show in Section~\ref{s-np} that it is \NP-complete to detect some kinds of windmill individually.

\section{Preliminaries}\label{s-pre}
 
Let $G=(V(G),E(G)$ be a graph.
The {\it neighbourhood} of a vertex $v\in V(G)$
is the set $N_G(v) = \{u \in V(G)\; |\; uv \in E(G)\}$, and the {\it degree} of $v$ is $|N_G(v)|$. 
We write $N[v] = N(v) \cup \{v\}$. For 
$S\subseteq V(G)$, we use
$N(S)=\cup_{u\in V}N(u)$. We write $G[S]$ for the subgraph of $G$ induced by $S$, and say that $S$ is {\it connected} if $G[S]$ is connected.
For $v \in V(G)$ and $S \subseteq V(G)$,
we say that $v$ is {\it complete} to $S$ if $S \subseteq N(v)$, and {\it anti-complete} to $S$ if $S \cap N(v) = \emptyset$. A set $T \subseteq V(G)$ is {\it (anti-)complete} to $S$ if every $v \in T$ is 
{\it (anti-)complete} to $S$. Two disjoint sets $S,T\subseteq V(G)$ are {\it adjacent} if there is an edge with one endpoint in $S$ and the other one in $T$.
Let $F$ be another graph. We write $F\subseteq G$ if $V(F)\subseteq V(G)$ and $E(F)\subseteq E(G)$.

If $u\in V(P)$ and $v\in V(P)$ for some path $P$, then $uPv$ denotes the subpath of $P$ from $u$ to $v$.
A {\it tripod} with {\it ends} $v_1$, $v_2$, $v_3$ is a tree of maximum degree~$3$ with exactly one vertex of degree~$3$ that has $v_1$, $v_2$ $v_3$ as its leaves.  
A {\it triangle tripod} with {\it ends} $v_1$, $v_2$, $v_3$ is obtained by taking three vertex-disjoint paths
$P_i$ with end-vertices $u_i$ and~$v_i$ (with possibly $u_i=v_i$) for $i\in \{1,2,3\}$ and adding the edges $u_1u_2$, $u_1u_3$ and $u_2u_3$.
Note that a triangle is a triangle tripod whose three vertices are the ends.

Let $G$ be a graph that contains a graph $H$ as an induced minor. A set ${\mathcal X} = \{X_y\; |\; y\in V(H)\}$ is an {\it induced minor model}, or more specifically, an {\it $H$-model} in $G$ if:

\begin{itemize}
\item [(i)] for every $y\in V(H)$, $X_y$ is non-empty;
\item [(ii)] for every $y,z\in V(H)$, $X_y\cap X_z=\emptyset$; and
\item [(iii)] for every $y,z\in V(H)$,  $X_y$ and $X_z$ are adjacent if and only if $yz\in E(H)$.
\end{itemize}

\noindent
For $y\in V(H)$, we write $G_y=G[X_y]$.
Note that we obtain $H$ by contracting edges along a spanning tree in every $G_y$ and deleting any vertex that is not in a bag $X_y$ for some $y\in V(H)$. 
Note also that a graph $G$ may have multiple $H$-models for some graph $H$. 

\begin{figure}[t]
    \centering    \includegraphics[width=0.7\linewidth,page=2]{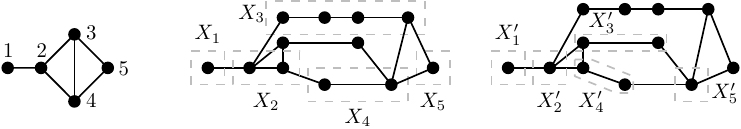}
    \caption{A graph $G$ with two bag-minimal $\mathbb{K}$-models, but only the right model is minimum.}
    \label{fig:minimum}
\end{figure}

For a graph $H$, let ${\cal X}$ and ${\cal X}'$ be  two $H$-models in a graph $G$. We say that ${\cal X}'$ is smaller than
${\cal X}$ if $\sum_{y\in V(H)}|X'_y|<\sum_{y\in V(H)}|X_y|$.
If there is no smaller model than ${\cal X}$, then ${\cal X}$ is a {\it minimum} $H$-model in $G$.
Moreover, ${\cal X'}$ is {\it properly contained} in ${\cal X}$ if for all $y\in V(H)$, it holds that $X'_y\subseteq X_y$. We say that ${\cal X}$ is \emph{bag-minimal} if
there exists no smaller $H$-model ${\mathcal X}'$ that is contained in ${\cal X}$.
While every minimum $H$-model is bag-minimal, the converse does not necessarily hold, see Figure~\ref{fig:minimum} for an example.

Let $G$ be a graph with an $H$-model ${\cal X}$ for some graph $H$.
For $yz\in E(H)$, a vertex in $X_y$ with a neighbour in $X_z$ is an 
{\it $X_z$-attachment (vertex)} in $X_y$. 

The next two lemmas are basic lemmas that are essentially known, but we need them in a specific form and therefore give their proofs as well.

\begin{lemma}\label{l:leaf}
For a graph $H$, let ${\cal X}$ be a 
bag-minimal $H$-model in a graph $G$. Let $y\in V(H)$, and let $T_y$ be a spanning tree $T_y$ of $G_y$. It holds that:
\begin{itemize}
    \item [(i)] $T_y$ has at most $|N_H(y)|$ leaves, and
    every leaf of $T_y$ is the unique $X_z$-attachment vertex in $X_y$ for some $z\in N_H(y)$; and
    \item [(ii)] every inner vertex of the (unique) path 
    from some vertex $u$ to some vertex $w$ in $T_y$ is the unique $X_z$-attachment vertex in $X_y$ for some $z\in N_H(y)$ if $uv\in E(G)$.
\end{itemize}
\end{lemma}

\begin{proof}
To prove (i), if a leaf $u$ of $T_y$ is not the unique $X_z$-attachment vertex in $X_y$ for some $z\in N_H(y)$, then the set
${\cal X}'$ with $X'_y=X_y\setminus \{u\}$ and $X'_z=X_z$ for all $z\in V(H)\setminus \{y\}$ is a smaller $H$-model in $G$ contained in ${\cal X}$, a contradiction. To prove~(ii), we argue the same.
\end{proof}

\begin{lemma}\label{l:general}
For a graph $H$ and a graph $G$ with 
a bag-minimal $H$-model ${\cal X}$, the following holds for every $y\in V(H)$: 
\begin{itemize}
\item [(i)] If $y$ has exactly neighbour $z$ in $H$, then $|X_y|=1$. 
\item [(ii)] If $y$ has exactly two neighbours $z_1$ and $z_2$ in $H$, then $G_y$ is a path (of possibly one vertex) from a unique $X_{z_1}$-attachment vertex to a unique $X_{z_2}$-attachment vertex in $X_y$.
\item [(iii)] If $y$ has exactly three neighbours $z_1$, $z_2$ and $z_3$ in $H$, then either:
\begin{itemize}
\item  $G_y$ is a path between a unique $X_{z_h}$-attachment and a unique $X_{z_i}$-attachment for some $h,i\in \{1,2,3\}$ with $h\neq i$; or
\item $G_y$ consists of a single vertex; or,
\item  $G_y$ is a tripod or triangle tripod whose ends are the unique $X_{z_i}$-attachment vertices for $i\in \{1,2,3\}$, respectively. 
\end{itemize}
In particular, if there is $j\in \{1,2,3\}$ with $|N(X_{z_j}) \cap X_y| \geq 2$, then $G_y$ is a path between a unique $X_{z_h}$-attachment and a unique $X_{z_i}$-attachment for $h,i\in \{1,2,3\}\setminus \{j\}$ with $h\neq i$.
\end{itemize}
\end{lemma}

\begin{proof}
To prove (i), we apply Lemma~\ref{l:leaf}-(i). To prove (ii), we first apply Lemma~\ref{l:leaf}-(i) to find that $G_y$ has a spanning path $T_y$ from a unique $X_{z_1}$-attachment vertex in $X_y$
to a unique $X_{z_2}$-attachment vertex in $X_y$. Next, we apply Lemma~\ref{l:leaf}-(ii) to find that $G_y=T_y$.

To prove (iii), assume that $G_y$ has more than one vertex. Let $T_y$ be a spanning tree of~$G_y$. 
We choose $T_y$ to have a vertex of degree at least~$3$
if $G_y$ has a vertex of degree at least~$3$; else $T_y$ is a path.
We distinguish between the following two cases:

\medskip
\noindent
{\bf Case 1.} $T_y$ has at least three leaves.\\
As $y$ has degree~$3$ in $H$, Lemma~\ref{l:leaf}-(i) tells us that
$T_y$ is a tripod whose three leaves $u_i$ are the unique $X_{z_i}$-attachment vertices for $i\in \{1,2,3\}$.
Let $v$ be the vertex of degree~$3$ in $T_y$.
Let $Q_i$ be the path from $u_i$ to $v$ in $T_y$.
 From Lemma~\ref{l:leaf}-(ii), it follows that $Q_1$, $Q_2$, $Q_3$ are induced paths in $G$. Suppose $e=st$ is an edge of $G$ with $s\in V(Q_1)$ and $t\in V(Q_2)$. If one of $s,t$, say $s$, is not a neighbour of $v$, then we can safely delete all inner vertices of $sQ_1v$, a contradiction. Suppose $s$ and $t$ are neighbours of $v$. 
 If one of $s$, $t$, say $s$, is adjacent to the neighbour of $u$ (if it exists) on $Q_3$, then we can safely delete $v$, another contradiction.
 Hence, $G$ is a tripod or triangle tripod with ends $u_1$, $u_2$, $u_3$.

\medskip
\noindent
{\bf Case 2.} $T_y$ has exactly two leaves.\\
From Lemma~\ref{l:leaf}-(ii), it follows that
$T_y$ 
 is a path between a unique $X_{z_h}$-attachment $u_1$ and a unique $X_{z_i}$-attachment $u_2$ for some $h,i\in \{1,2,3\}$ with $h\neq i$.
By our choice of $T_y$ and since $G_y$ is connected, we find that $G_y=T_y$, or $G_y$ is a cycle obtained from $T_y$ after adding the edge $u_1u_2$.
In the first case we are done. In the second case, it follows from the 
bag-minimality of ${\cal X}$ that $G_y$ is a triangle whose vertices are the unique $X_{z_i}$-attachment vertices in $X_y$, which means we are also done.

\medskip
\noindent
If there is some $j\in \{1,2,3\}$ such that $|N(X_{z_j}) \cap X_y| \geq 2$, then there is no unique $X_{z_j}$-attachment in $X_y$. It follows that in this case $G_y$ is a path between the unique $X_{z_h}$-attachment and the unique $X_{z_i}$-attachment for $h, i \in \{1,2,3\} \setminus \{j\}$ with $h\neq i$.
\end{proof}

\noindent
For fixed $k\geq 2$, the {\sc $k$-Disjoint Subgraphs} problem has as input a graph $G$ with $k$ pairwise disjoint vertex subsets $Z_1,\ldots, Z_k\subseteq V(G)$. It asks whether there exist $k$ pairwise disjoint connected vertex subsets $S_1,\ldots,S_k$ with $Z_i\subseteq S_i$ for every $i\in \{1,\ldots,k\}$. 
For a constant $c\geq 1$, a subclass of instances $(G,Z_1,\ldots,Z_k)$ is {\it $c$-bounded} if $|Z_i|\leq c$ for every $i\in \{1,\ldots,k\}$.
The {\sc $k$-Disjoint Subgraphs} problem is \NP-complete even if $k=2$ and $|Z_1|=2$~\cite{HPW09}, but polynomial-time solvable for $c$-bounded instances, as proven by Robertson and Seymour.

\begin{theorem}[\cite{RS95}]\label{t-kdisjoint}
For every pair of constants $c\geq 1$ and $k\geq 2$, {\sc $k$-Disjoint Subgraphs} is polynomial-time solvable for $c$-bounded instances.
\end{theorem}

\section{The Kite $\mathbb{K}$}\label{s-diamondplus}

To prove that ${\mathbb{K}}$-{\sc Induced Minor} is polynomial-time solvable, we first show a structural lemma regarding $\mathbb{K}$-models in $G$.

\begin{lemma}\label{l:diamond-bags}
 If a graph $G$ contains $\mathbb{K}$ as an induced minor, then 
$G$ has a $\mathbb{K}$-model ${\mathcal X}$ such that: 
        $|X_1| = |X_5| = 1$, and 
        $|N(X_1) \cap X_2| \leq 3$, and
        $|N(X_5) \cap X_3| \leq 2$, and $|N(X_5) \cap X_4| \leq 2$.
\end{lemma}

\begin{proof}
Let $G$ be a graph that contains $\mathbb{K}$ as an induced minor.
Hence, $G$ has a $\mathbb{K}$-model $\mathcal X$. 
We choose $\mathcal X$ to be a minimum $\mathbb{K}$-model that, over all minimum $\mathbb{K}$-models in $G$, minimises $|X_5|$ and subject to this, minimises $|N(X_5) \cap X_3| + |N(X_5) \cap X_4|$. 
As $\mathcal X$ is minimum, ${\cal X}$ is bag-minimal. Hence, by Lemma~\ref{l:general}, $|X_1| = 1$. Let $u$ be the unique vertex in $X_1$.

We first claim that $|N(u) \cap X_2| \leq 3$. For a contradiction, assume $|N(u) \cap X_2| \geq 4$. By Lemma~\ref{l:general}, this means that $G_2$ is a path from a vertex $a$ to a vertex $b$, where $a$ is the unique $X_3$-attachment vertex and $b$ is the unique $X_4$-attachment vertex in $X_2$.

Let $r_1\in X_2$ be the neighbour of $u$ closest to $a$ in $G_2$, and let 
 $r_2\in X_2$ be the neighbour of $u$ closest to $b$ in $G_2$; note that possibly $r_1 = a$ or $r_2 = b$. Let $r_1,q_1,q_2$ be the first three vertices on the subpath $r_1G_2r_2$. Since $|N(u) \cap X_2| \geq 4$, the vertices $r_1, q_1, q_2, r_2$ are distinct. It follows that there is a $\mathbb{K}$-model  ${\cal X}'$ in $G$ such that $X'_1 = \{q_1\}$, $X'_2 = (X_2 \setminus \{q_1,q_2\}) \cup \{u\}$, $X'_3 = X_3$, $X'_4 = X_4$ and $X'_5 = X_5$; see also Figure~\ref{fig:diamond-bags-X2}. As $q_2$ is not in any bag of ${\cal X}'$, ${\cal X}'$ is smaller than ${\cal X}$; a contradiction. Hence, $|N(u) \cap X_2| \leq 3$.

 We now claim that $|X_5|\geq 2$. By Lemma~\ref{l:general}, $G_5$ is a path from a unique $X_3$-attachment vertex $v$ to a unique $X_4$-attachment vertex $v'$ in $X_5$. If $v' \neq v$, then we obtain a new $\mathbb{K}$-model ${\cal X}'$ in $G$ after moving $v'$ from $X_5$ to $X_4$. Let ${\cal X}'$. As $|X_1'| =1$ and $|X_5'| < |X_5|$, this contradicts our choice of $\mathcal X$. It follows that $v' = v$ and $X_5=\{v\}$. That is, $|X_5|=1$.

 It remains to show that $|N(v) \cap X_3| \leq 2$ and $|N(v) \cap X_4| \leq 2$. Suppose, this is not the case. As, without loss of generality, $|N(v) \cap X_3| \geq |N(v) \cap X_4|$, it follows that, $|N(v) \cap X_3| \geq 3$. 
  By Lemma~\ref{l:general}, $G_3$ is a path from a unique $X_2$-attachment vertex $a$ to a unique $X_4$-attachment vertex $b$ in $X_3$.
        Let $r_1\in X_3$ be the neighbour of $v$ closest to $a$ in $G_3$, and let $r_2\in X_3$ be the neighbour of $v$ closest to $b$ in $G_3$; note that possibly $r_1 = a$ or $r_2 = b$. Let $r_1,q_1, q_2$ be the first three vertices on the subpath $r_1G_3r_2$. Since $|N(v) \cap N(X_3)| \geq 3$, the vertices $r_1$, $q_1$, $r_2$ are distinct. We now set
        $X_1'=X_1$, $X_2'=X_2$,
      $X_3' = V(aX_3r_1) \cup \{v\}$, $X_4' = V(q_2X_3b) \cup X_4$ and $X_5' = \{q_1\}$.
This yields another minimum $\mathbb{K}$-model ${\cal X}'$ in $G$ with $|X_5'|=1$;      
see also Figure~\ref{fig:diamond-bags-H5}. 
Note that $N(q_1) \cap X_3'\subseteq \{r_1,v\}$ and $N(q_1) \cap X_4'=\{q_2\}$ and so $|N(X_5)'\cap X'_3|+|N(X_5')\cap X_4'| \leq 3$.
As $|N(X_5)\cap X_3|+|N(X_5)\cap X_4|=
|N(v) \cap X_3| + |N(v) \cap X_4| \geq 4$, this contradicts our choice of ${\cal X}$. Hence, $|N(v) \cap X_3| \leq 2$ and $|N(v) \cap X_4| \leq 2$. This concludes the proof.
\end{proof}

\begin{figure}
    \centering
    \includegraphics[width=0.8\linewidth, page=5]{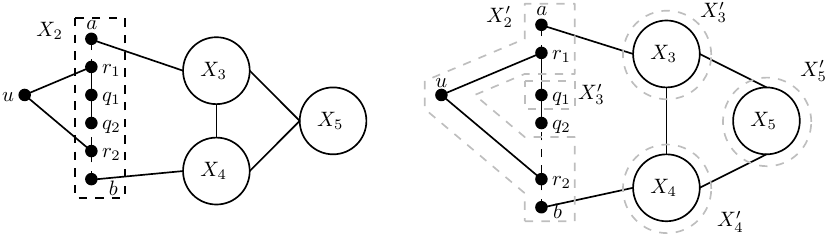}
    \caption{The first ``swap'' in the proof of Lemma~\ref{l:diamond-bags}.
    Note that $q_1$ and $q_2$ may be adjacent to $u$, and that $q_2$ is no longer in any bag in the right $\mathbb{K}$-model in $G$.}
    \label{fig:diamond-bags-X2}
\end{figure}

\begin{figure}
    \centering
    \includegraphics[width=0.7\linewidth, page=6]{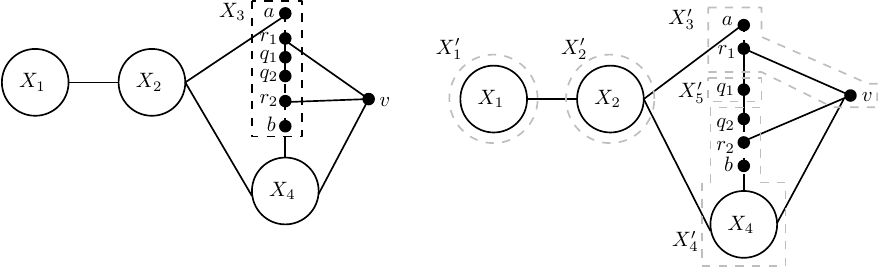}
    \caption{The second ``swap'' in the proof of Lemma~\ref{l:diamond-bags}. Note that $q_1$ and $q_2$ may be adjacent to $v$.}
    \label{fig:diamond-bags-H5}
\end{figure}

\begin{figure}
    \centering
    \includegraphics[width=0.4\linewidth, page=7]{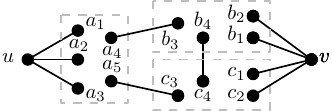}
    \caption{Those vertices ``guessed'' in the algorithm of Theorem~\ref{t-dplus}.}
    \label{fig:diamond-guess}
\end{figure}

\noindent
We now give our polynomial-time algorithm. The main idea is that we first reduce to a polynomial number of $5$-bounded instances of $3$-{\sc Disjoint Subgraphs} and then solve each of these instances by applying Theorem~\ref{t-kdisjoint}.

\begin{theorem}\label{t-dplus}
$\mathbb{K}$-{\sc Induced Minor} is polynomial-time solvable.
\end{theorem}

\begin{proof}
Let $G$ be a graph.
From Lemma~\ref{l:diamond-bags} it follows that $G$ contains $\mathbb{K}$ as an induced minor if and only if $G$ contains some $\mathbb{K}$-model ${\mathcal X}$ such that $|X_1| = |X_5| = 1$, and $|N(X_1) \cap X_2| \leq 3$, and $|N(X_5) \cap X_3| \leq 2$, and $|N(X_5) \cap X_4| \leq 2$.
Hence, it suffices to suffices to decide if $G$ contains a $\mathbb{K}$-model of this form.
Our algorithm will proceed by branching over all options for the sets $X_1$, $X_5$, $N(X_1) \cap X_2$, $N(X_5) \cap X_3$ and $N(X_5) \cap X_4$ as well as three edges that will connect the sets $X_2$, $X_3$, $X_4$. We then decide if there exists a $\mathbb{K}$-model ${\cal X}$ in $G$ such that these vertices/edges are contained in those bags as described.

Namely we consider all possible $O(n^{15})$ options of choosing a vertex $u$ to form $X_1$; 
a vertex $v\notin N(u)$ to form $X_5$; at most three vertices $a_1,a_2,a_3$ of $N(u) \setminus N(v)$ to go into $X_2$; at most two vertices $b_1,b_2$ of 
$N(v) \setminus N(u)$ to go into $X_3$; at most two vertices $c_1,c_2$ of $N(v) \setminus N(u)$
to go into $X_4$; an edge $a_4b_3$ with 
$a_4\notin N(v)$ and $b_3\notin N(u)$ to go into $X_2$ and $X_3$, respectively;
an edge $a_5c_3$ with $a_5\notin N(v)$ and 
$c_3\notin N(u)$ to go into $X_2$ and $X_4$ respectively; and
an edge $b_4c_4$ with $b_4\notin N(u)$ and 
$c_4\notin N(u)$ to go into $X_3$ and $X_4$, respectively.
We require that the sets $\{a_1,\ldots,a_5\}$, $\{b_1,\ldots,b_4\}$, $\{c_1,\ldots,c_4\}$ are pairwise disjoint.
However, we not require that the vertices $a_1,\ldots,a_5$ are pairwise distinct, likewise we allow duplications in the vertices $b_1,\ldots,b_4$ and the vertices $c_1,\ldots,c_4$. See Figure~\ref{fig:diamond-guess}.

Given that vertex~$1$ is anti-complete to $\{3,4,5\}$ in $\mathbb{K}$, the sets $X_2$, $X_3$ and $X_5$ cannot contain any vertex of $N[u] \cup N[v]$. We may therefore consider the graph $G'$ obtained from $G$ by removing all vertices of $N[u] \cup N[v]$.
 We set $Z_1=\{a_1,\ldots,a_5\}$, 
$Z_2=\{b_1,\ldots,b_4\}$ and $Z_3=\{c_1,\ldots,c_4\}$. 
We highlight that $(G',Z_1,Z_2,Z_3)$ is a $5$-bounded instance of {\sc $3$-Disjoint Connected Subgraphs}.
Hence, by Theorem~\ref{t-kdisjoint}, we find in polynomial time whether $V(G')$ contains three pairwise disjoint connected sets $S_1$, $S_2$, $S_3$ with
$Z_i\subseteq S_i$ for every $i\in \{1,2,3\}$.
If such sets exist, then we let $X_1=\{u\}$, $X_2=S_1$, $X_3=S_2$, $X_4=S_3$ and $X_5=\{v\}$. By construction, ${\cal X}$ is a $\mathbb{K}$-model in $G$.
Otherwise, we discard our guess,
as there is no $\mathbb{K}$-model 
in $G$ such that those guessed vertices are contained in the sets as described.

The correctness of our algorithm follows from its description. Since we apply Theorem~\ref{t-kdisjoint} $O(n^{15})$ times, our algorithm runs in polynomial time. 
\end{proof}

\section{The Graph $\mathbf{F_1}$}\label{s-f1}

To prove that $F_1$-{\sc Induced Minor} is polynomial-time solvable, we show a structural lemma regarding $F_1$-models.

\begin{lemma}\label{l:loli}
For every minimum $F_1$-model ${\cal X}$ in a graph $G$, $X_1=\{u\}$ for some $u\in V(G)$, and $X_2$ contains vertices of at most seven connected components of $G[N(u)]$.
\end{lemma}

\begin{proof}
Let ${\cal X}$ be a minimum $F_1$-model in a graph $G$. Since $\mathcal X$ is minimum, ${\cal X}$ is bag-minimal, so from Lemma~\ref{l:general}, $X_1=\{u\}$ for some $u\in V(G)$. It remains to show that $X_2$ contains vertices of at most seven connected components of $G[N(u)]$.
For a contradiction, suppose there is some set ${\cal D}$ of eight connected components of $G[N(u)]$, such that each component in ${\cal D}$ contains at least one vertex from $X_2$.

Let $T$ be a spanning tree of $G_2$. As vertex $2$ is complete to $\{1,3,4,5\}$  in $F_1$, and ${\cal X}$ is bag-minimal, 
we find that $T$ has at most four leaves
by Lemma~\ref{l:leaf}-(i).
If $T$ has four leaves, then one of them is the only $X_1$-attachment vertex in $X_2$, again due to Lemma~\ref{l:leaf}-(i). This is not possible, as $u$ has at least eight neighbours in $X_2$. Hence, $T$ has at most three leaves. By the pigeonhole principle,
$T$ has a path $P$ from a leaf~$r_1$ to a leaf $r_2$ of $T$ that contains vertices from at least five different connected components of ${\cal D}$.
From Lemma~\ref{l:leaf},  $r_1$ is the unique $X_i$-attachment vertex and $r_2$ is the unique $X_j$-attachment vertex
for some $i,j\in \{3,4,5\}$ with $i\neq j$. 
By symmetry, we may assume $i=3$ and $j=4$. 

Let $a$ and $b$ be those neighbours of $u$ on $P$ closest to $r_1$ and $r_2$, respectively. Let $a'$ be the neighbour of $a$ on $aPb$, and let $b'$ be the neighbour of $b$ on $aPb$. If $a'$ is adjacent to $u$, then $a$ and $a'$ are in the same component $D\in {\cal D}$.
The same holds for $b$ and $b'$. Hence, $a'Pb'$ has vertices from at least five different connected components of ${\cal D}$. Any subpath of $P$ between two vertices of different components of ${\cal D}$ must contain at least one vertex not in $N(u)$. It follows that $|V(a'Pb')|\geq 9$. Let $c$ be the neighbour of $a'$ on $a'Pb'$, and $c'$ be the neighbour of $c$ on $cPb'$. Let $d$ be the neighbour of $b'$ on $a'Pb'$, and let $d'$ be the neighbour of $d$ on $a'Pd$.
As $|V(a'Pb')|\geq 9$, we have that the vertices $a,a',c,c',d',d,b',b$ are pairwise distinct.

\begin{figure}[b]
    \centering
    \includegraphics[width=0.7\linewidth, page=8]{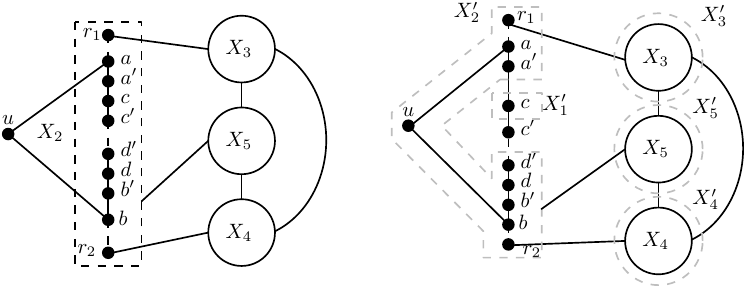}
\caption{The first ``swap'' in the proof of Lemma~\ref{l:loli}. Note that $c'$ is not used in the right $F_1$-model.}\label{fig:loli-X5-c}
\end{figure}

First suppose that $\{c,c'\}$ is anti-complete to $X_5$.
Hence, there is an $X_5$-attachment vertex in $X_2\setminus \{c,c'\}$.  Moreover, $X_2\setminus \{c,c'\}$ contains $r_1$ and $r_2$ and thus has both the unique $3$-attachment vertex and the unique $4$-attachment vertex in $X_2$.
It follows that $c$ is anti-complete to $X_5$, and so $c$ only has neighbours in $X_2\cup \{u\}$. We also note that $(X_2\setminus \{c,c'\})\cup \{u\}$ is connected. Hence, $G$ contains an $F_1$-model ${\cal X}'$ such that
 $X_1'=\{c\}$, $X_2'=(X_2\setminus \{c,c'\})\cup \{u\}$, $X_3'=X_3$, $X_4'=X_4$ and $X_5'=X_5$. See Figure~\ref{fig:loli-X5-c}.
 Because $c'$ is not in any bag of ${\cal X}'$, it holds that ${\cal X}'$ is smaller than ${\cal X}$, a contradiction. 
 
From the above, we conclude that at least one of $c$ and $c'$ is an $X_5$-attachment vertex in $X_2$. By symmetry, at least one of $d$ and $d'$ is an $X_5$-attachment vertex in $X_2$. Since $X_2$ has at least two $X_5$-attachment vertices and also at least seven $X_1$-attachment vertices, $T$ has only two leaves by Lemma~\ref{l:leaf}~(i), and thus $T$ is a path. In fact, by
 Lemma~\ref{l:leaf}~(ii), $G_2=T$ is a path (which goes from $r_1$ to $r_2$ and contains the vertices $a, a', c, c', d', d, b', b$). 

\begin{figure}[t]
    \centering
    \includegraphics[width=0.7\linewidth, page=9]{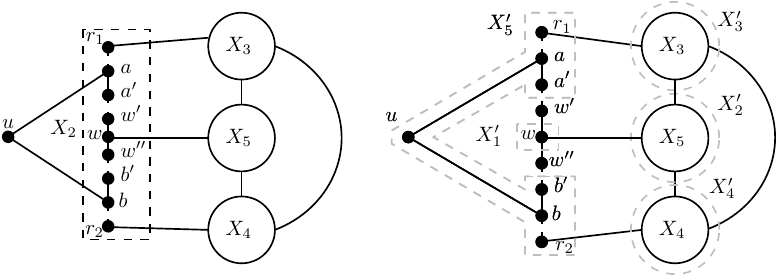}
    \caption{The second ``swap'' in the proof of Lemma~\ref{l:loli}. Note that neither $w'$ nor $w''$ is used in the right $F_1$-model.}\label{fig:loli-X5-w}
\end{figure}

\begin{figure}
    \centering
    \includegraphics[width=0.7\linewidth, page=10]{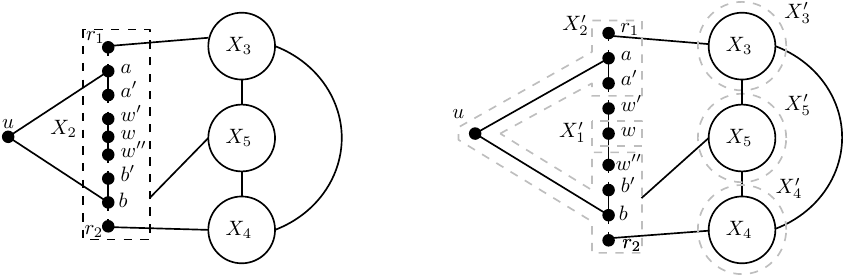}
    \caption{The third ``swap'' in the proof of Lemma~\ref{l:loli}. Note $w'$ is not used in the right $F_1$-model.}\label{fig:loli}
\end{figure}

Let $W$ be the set of vertices of $cPd$ not adjacent to $u$. Since $a'Pb'$ has vertices from different components of ${\cal D}$, it holds that $W$ is non-empty.
Now let $w$ be some arbitrary vertex in $W$, and let $w'$ and $w''$ be the two neighbours of $w$ on $a'Pb'$.
Since $|V(a'Pb')|\geq 9$, we may assume, by symmetry, that $\{c,c'\}\cap \{w,w',w''\}=\emptyset$.
Hence, as one of $c,c'$ is an $X_5$-attachment vertex in $X_2$, we find that there is an $X_5$-attachment vertex in $X_2\setminus \{w,w',w''\}$. Moreover, $X_2\setminus \{w,w',w''\}$ contains the unique $X_3$-attachment vertex $r_1$ and the unique $X_4$-attachment vertex $r_2$ in $X_2$. 
As $w\in W$ and $X_2$ is a path, $w$ only has neighbours in $X_5\cup \{w',w''\}$.
Finally, observe that $(X_2\setminus \{w,w',w''\})\cup \{u\}$ is connected.
It therefore follows that if $w$ is adjacent to some vertex in $X_5$, then $G$ contains an $F_1$-model ${\cal X}'$ such that
$X_1'=\{w\}$, $X_2'=X_5$, $X_3'=X_3$, $X_4'=X_4$ and  $X_5'=(X_2'\cup \{u\})\setminus \{w,w',w''\}$. See Figure~\ref{fig:loli-X5-w}.
As $w'$ and $w''$ are not in any bag of ${\cal X}'$, we find that
${\cal X'}$  is smaller than ${\cal X}$; a contradiction. Hence, $w$ is anti-complete to $X_5$.
We now find that $G$ contains an $F_1$-model ${\cal X}'$ such that
 $X_1'=\{w\}$, $X_2'=(X_2\setminus \{w,w'\})\cup \{u\}$, $X_3'=X_3$, $X_4'=X_4$ and $X_5'=X_5$. See Figure~\ref{fig:loli}.
 Because $w'$ is not in any bag of ${\cal X}'$, it holds that ${\cal X}'$ is smaller than ${\cal X}$, a contradiction. This completes the proof of the lemma.
\end{proof}

\noindent
We are now ready to give our algorithm.
In this case, we reduce to a polynomial number of $15$-bounded instances of $4$-{\sc Disjoint Subgraphs}.

\begin{theorem}\label{t-f1}
$F_1$-{\sc Induced Minor} is polynomial-time solvable.
\end{theorem}

\begin{proof}
Let $G$ be a graph. From Lemma~\ref{l:loli}, it follows that if $G$ has $F_1$ as an induced minor, then $G$ has an $F_1$-model ${\cal X}$ where $X_1=\{u\}$ for some $u\in V(G)$ and $X_2$ contains vertices from at most seven connected components of $G[N(u)]$. We consider all $O(n)$ options for choosing the vertex $u$ that forms $X_1$. We also consider all $O(n^7)$ options of choosing the set~${\cal D}$ of connected components of $G[N(u)]$ that have a nonempty intersection with $X_2$. 

For each choice of $u$ and set of connected components ${\cal D}=\{D_1,\ldots,D_s\}$ for some $s\leq 7$, we do as follows.
First, as vertex~$2$ dominates $F_1$, we observe that if $u$ has neighbours $v$ and $w$, then $G$ has an $F_1$-model~${\cal X}$ with $X_1=\{u\}$ and $v\in X_2$ if and only if $G$ has an $F_1$-model ${\cal X}$ with $X_1=\{u\}$ and $\{v,w\}\subseteq X_2$. 
Let $G'$ be that graph obtained from $G$ by contracting the set $V(D_i)$ to a singleton vertex $f_i$, for $i \in \{1, \ldots, s\}$. Let $F = \{f_1, \ldots, f_{s}\}$.
We note that $G'$ has an $F_1$-model ${\cal X}$ such that $X_1=\{u\}$ and $F\subseteq X_2$ if and only if $G$ has an 
$F_1$-model~${\cal X}^*$ such that $X_1^*=\{u\}$ and $X_2^*$ has a non-empty intersection with every component in~${\cal D}$. That is we now decide if $G'$ contains such an $F_1$-model ${\cal X}$.

We consider all $O(n^{12})$ options of choosing vertices/edges in $G'$ with the following properties. 
We choose edges $a_1b_1$, $a_2c_1$, $a_3d_1$, $b_2c_2$, $b_3d_2$ and $c_3d_3$ such that: 
the vertices $a_1,a_2,a_3$ will be contained in $X_2$;
the vertices $b_1,b_2,b_3 \notin N(u)$ and will be contained in $X_3$;
the vertices $c_1,c_2,c_3 \notin N(u)$ and will be contained in $X_4$; and
the vertices $d_1,d_2,d_3 \notin N(u)$ and will be contained in $X_5$.
We allow duplications within each of the sets $\{a_1,a_2,a_3,f_1,\ldots,f_s\}$, $\{b_1,b_2,b_3\}$, $\{c_1,c_2,c_3\}$ and $\{d_1,d_2,d_3\}$.

We remove all neighbours of $u$ not either in $F$ or guessed above, as none of these vertices can be contained the bags $X_2$, $X_3$ and $X_5$, as $1$ is anti-complete to $\{3,4,5\}$ in $H$.
Let $G''$ be the resulting graph. We set $Z_1=\{a_1,a_2,a_3,f_1,\ldots,f_s\}$, $Z_2=\{b_1,b_2,b_3\}$, $Z_3=\{c_1,c_2,c_3\}$ and $Z_4=\{d_1,d_2,d_3\}$.
This yields a $15$-bounded instance $(G'',Z_1,Z_2,Z_3,Z_4)$ of {\sc $4$-Disjoint Connected Subgraphs}.
By Theorem~\ref{t-kdisjoint}, we find in polynomial time whether $V(G'')$ contains four pairwise disjoint connected sets $S_1,\ldots, S_4$ with
$Z_i\subseteq S_i$ for every $i\in \{1,\ldots,4\}$. 
If such sets exist, then we find an $F_1$-model ${\cal X}$ in $G'$ such that $X_1=\{u\}$, $X_2=S_1$, $X_3=S_2$, $X_4=S_3$ and $X_5=S_4$. As observed above, it follows that $G$ contains $F_1$ as an induced minor.
Otherwise, we discard our guess, as there is no $F_1$-model in $G''$ such that those guessed vertices are contained in the sets as described.

The correctness of our algorithm follows from its description. Since we apply Theorem~\ref{t-kdisjoint} $O(n^{19})$ times, our algorithm runs in polynomial time. 
\end{proof}

\section{The Graph $\mathbf{F_2}$}\label{s-f2}

To prove that $F_2$-{\sc Induced Minor} is polynomial-time solvable, we show a structural lemma.

\begin{lemma}\label{l:k4-sub}
If a graph $G$ contains $F_2$ as an induced minor, then 
$G$ has an $F_2$-model ${\mathcal X}$ such that:
\begin{itemize}
\item [(i)] $X_1=\{u\}$ for some $u\in V(G)$; 
\item [(ii)] $X_2=\{r\}$ for some $r\in V(G)$;
\item [(iii)] $u$ has at most two neighbours in $X_5$
\item [(iv)] $r$ has at most two neighbours in $X_3$ and at most two neighbours in $X_4$.
\end{itemize}
\end{lemma}

\begin{proof}
Let $G$ be a graph that contains $F_2$ as an induced minor.
Let ${\cal X}$ be a minimum $F_2$-model in a graph $G$ that, over all minimum $F_2$-models in $G$,
minimises $|X_1|$ and subject to this, minimises 
$|X_2|$ and subject to this, minimises $|X_5|$. As ${\cal X}$ is minimum,
${\cal X}$ is bag-minimal. As the only neighbours of vertex~$1$ in $F_2$ are $2$ and $5$, we find that $G_1$ is a path from a unique $X_2$-attachment vertex $u$ to a unique $X_5$-attachment vertex $u'$ in $X_1$ by Lemma~\ref{l:leaf}. If $u\neq u'$,
then moving $X_1\setminus \{u\}$ to $X_5$ yields a minimum $F_2$-model ${\cal X}'$ in $G$ with $|X_1'|<|X_1|$, a contradiction. Hence, $X_1=\{u\}$, which proves (i).

We now consider $X_2$.
We first claim that there are at most two unique attachment vertices in $X_2$. Suppose for contradiction that there are three. Let $u_1$ be the unique $X_1$-attachment vertex, $u_2$ be the unique $X_3$-attachment vertex, and $u_3$ be the unique $X_4$-attachment vertex in $X_2$. Since $u_2$ is anti-complete to $X_1$, 
moving $u_2$ from $X_2$ to $X_3$ yields a minimum $F_2$-model~${\cal X}'$ in $G$ with $|X_1'|=|X_1|$ but
$|X_2'|<|X_2|$, contradicting our choice of ${\cal X}$. Hence, there are at most two unique attachment vertices in $X_2$. We now apply Lemma~\ref{l:leaf} to find that every spanning tree $T$ of $G_2$ is a path. That is, $G_2$ is either a path or a cycle. If $G_2$ is a cycle, then for each $\ell \in X_2$, there is a spanning tree of $G_2$, in which $\ell$ is a leaf. From Lemma~\ref{l:leaf} it now follows that $\ell$ is a unique attachment vertex in $X_2$. As $G_2$ is a cycle, $|X_2| \geq 3$; a contradiction, since $X_2$ has at most two unique attachment vertices. Hence, $G_2$ is a path.

\begin{figure}[b]
    \centering
    \includegraphics[width=0.7\linewidth, page=11]{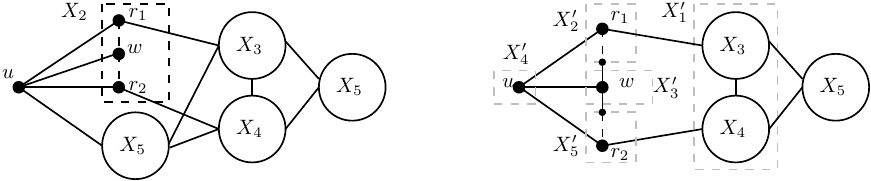}
    \caption{The first ``swap'' in the proof of Lemma~\ref{l:k4-sub}. Note $X_5$ is not present in the right model.}
    \label{fig:k4-at-most-2}
\end{figure}

Let $r_1$ and $r_2$ be the two ends of $G_2$ (with possibly $r_1=r_2$).
We say that $X_2$ has {\it type 1} if one of $\{r_1, r_2\}$ is the unique $X_1$-attachment vertex in $X_2$, and $X_2$ has {\it type 2} otherwise. In what follows we analyse the structure of $X_2$ in both cases.
By Lemma~\ref{l:leaf}, we may assume that without loss of generality one of the following holds:

\begin{itemize}
\item $r_1$ is the unique $X_1$-attachment vertex in $X_2$, and $r_2$ is the unique $X_j$-attachment vertex in $X_2$ for some $j\in \{3,4\}$, say, $j=3$, or
\item $r_1$ is the unique $X_3$-attachment vertex in $X_2$, and $r_2$ is the unique $X_4$-attachment vertex in $X_2$.
\end{itemize}

\noindent
That is, if $X_2$ has type 1, then $r_2$ is the unique $X_3$-attachment vertex in $X_2$. Moreover, $r_2$ is also the unique $X_4$-attachment vertex in $X_2$, else we can move $r_2$ from $X_2$ to $X_3$ to obtain a minimum $F_2$-model 
${\cal X}'$ in $G$ with $|X_1'|=|X_1|$ but $|X_2'|<|X_2|$, contradicting our choice of~${\cal X}$.  

If $X_2$ has type 2 then it follows from the above that without loss of generality, $r_1$ is the unique $X_3$-attachment vertex and $r_2$ is the unique $X_4$-attachment vertex in $X_2$. If $r_1$ is not an $X_1$-attachment vertex, then moving $r_1$ from $X_2$ to $X_3$ yields a minimum $F_2$-model~${\cal X}'$ in $G$ with $|X_1'|=|X_1|$ but
$|X_2'|<|X_2|$, contradicting our choice of ${\cal X}$. Hence, $r_1$, and symmetrically, $r_2$ are both $X_1$-attachment vertices. That is, $r_1,r_2 \in N(u)$.

We now claim that if $X_2$ has type 2, then $N(u)= \{r_1,r_2\}$.
Suppose that $u$ has some third neighbour $w$ in $X_2$. Let $X_1=X_3\cup X_4$, $X_2' = V(r_1X_2w) \setminus \{w\}$, $X_3' = \{w\}$, $X_4' = \{u\}$ and $X_5' = V(wX_2r_2)\setminus \{r_2\}$. As $w$ is not adjacent to any vertex in $X_3 \cup X_4$, this yields another $F_2$-model ${\cal X}'$ in $G$;
see Figure~\ref{fig:k4-at-most-2}. Since no bag of $\mathcal{X}'$ contains a vertex from $X_5$, we find that ${\cal X}'$ is smaller than ${\cal X}$, a contradiction. We conclude that $r_1$ and $r_2$ are the only neighbours of $u$ in $X_2$. 

\begin{figure}
    \centering
    \includegraphics[width=0.7\linewidth, page=12]{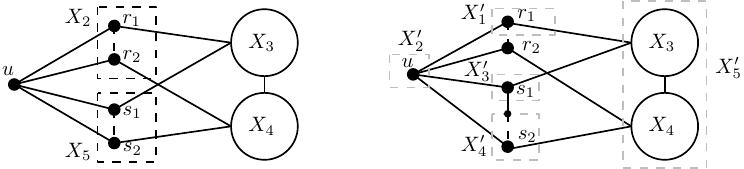}
    \caption{The second ``swap'' in the proof of Lemma~\ref{l:k4-sub}. Note $r_2$ is not used in the right model.}
    \label{fig:k4-2-then-1}
\end{figure}

By symmetry, we can also say that $X_5$ has either type~1 or~2. So, in particular, $X_5$ is also a path, say with ends $s_1$ and $s_2$. 
Suppose both $X_2$ and $X_5$ are of type~2, meaning that $N(u)\cap X_2=\{r_1,r_2\}$ and $N(u)\cap X_5=\{s_1,s_2\}$. We set $X_1' = \{r_1\}$, $X_2' = \{u\}$, $X_3' = \{s_1\}$, $X_4' = X_5 \setminus \{s_1\}$, $X_5' = X_3 \cup X_4$. 
 This yields another $F_2$-model ${\cal X}'$ in $G$;
 see Figure~\ref{fig:k4-2-then-1}. Since $r_2$ is not in any bag of ${\cal X}'$, we find that ${\cal X}'$ is smaller than ${\cal X}$, a contradiction.
We conclude that one of $X_2, X_5$, say $X_2$, is of type~1. As ${\cal X}$ minimises $|X_2|$ with priority, $X_2$ has type 1. Further, since $X_5$ is either of type~1 or of type~2, our analysis above gives that $u$ has at most two neighbours in $X_5$. Hence, we have proven (iii).
 
 We will now prove (ii). As $X_2$ is of type~1, $G_2$ is a path from $r_1$ to $r_2$, where $r_1$ is the unique $X_1$-attachment vertex and $r_2$ is the unique $X_j$-attachment vertex for each $j\in \{3,4\}$. If $r_1 \neq r_2$, then let $r_2'$ be the neighbour of $r_2$ in $G_2$. After moving the vertices $(V(X_2)\setminus \{r_2,r_2'\})\cup \{u\}$ to $X_5$ and replacing $u$ by $r_2'$, we obtain a new $F_2$-model ${\cal X}'$ such that $|X_1'| = 1$ and $|X_2'| < |X_2|$, thus contradicting our choice of ${\cal X}$. Hence, $r_1 = r_2$. This proves~(ii).

 \begin{figure}
    \centering
    \includegraphics[width=0.45\linewidth, page=13]{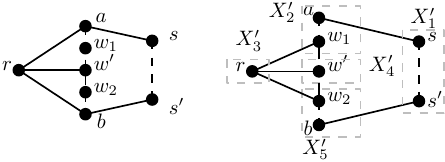}
    \caption{The third swap in the proof of Lemma~\ref{l:k4-sub}. Note $u$ is not used in the right $F_2$-model.}
    \label{fig:k4-small-nei}
\end{figure}

 Finally, we prove (iv). Let $s$ be an $X_3$-attachment vertex in $X_5$, say $s$ is adjacent to $a\in X_3$. Let $s'$ be an $X_4$-attachment vertex in $X_5$, say $s'$ is adjacent to $b\in X_4$. Let $P$ be a shortest path from $a$ to $b$ in $G[X_3\cup X_4]$. Note that possibly $s = s'$. We further choose $a$ and $b$ such that $\{s,s'\}$ is anti-complete to the internal vertices of $P$.
  We now claim that $r$ has at most two neighbours on $P$. For a contradiction, assume $r$ has neighbours $w_1, w', w_2$ on $P$ in this order. That is, $w'$ lies on the subpath $w_1Pw_2$. We set $X_1' = X_5$, $X_2' = V(aPw' )\setminus \{w'\}$, $X_3' = \{r\}$, $X_4' = \{w'\}$ and $X_5' = V(w'Pb) \setminus \{w'\}$. This yields another $F_2$-model ${\cal X}'$ in $G$; 
see Figure~\ref{fig:k4-small-nei}.
Since $u$ is not in any bag of $\mathcal{X}'$, we find that ${\cal X}'$ is smaller than ${\cal X}$, a contradiction. Hence, $r$ has at most two neighbours on $P$. As $\cal X$ is minimum, if $N(r) \cap V(P) \cap X_3 \neq \emptyset$ and $N(r) \cap V(P) \cap X_4 \neq \emptyset$, then $X_3 \cup X_4 = V(P)$ and so, in this case, (iv) holds. This leaves the case where, without loss of generality, $N(r) \cap V(P) \cap X_4 = \emptyset$. Let $w_3$ be a neighbour of $r$ closest to $P$ in $G_4$. We then take a shortest path $Q$ from $w_3$ to a vertex on $P$. As $\cal X$ is minimum, $X_3\cup X_4=V(P)\cup V(Q)$, hence, (iv) holds, and so we have proven the lemma.
\end{proof}

\noindent
We reduce again to a polynomial number of $4$-bounded instances of $3$-{\sc Disjoint Subgraphs}.

\begin{theorem}\label{t-f1}
$F_2$-{\sc Induced Minor} is polynomial-time solvable.
\end{theorem}

\begin{proof}
Let $G$ be a graph. From Lemma~\ref{l:k4-sub}, it follows that if $G$ has $F_2$ as an induced minor, then $G$ has an $F_2$-model ${\cal X}$ with properties (i)--(iv). We therefore consider all $O(n^{12})$ options for choosing the vertex $u$ that forms $X_1$; a neighbour~$r$ of $u$ that forms $X_2$, 
at most two vertices $s_1,s_2 \in N(u) \setminus N(r)$ that will be contained in $X_5$; 
at most two vertices $a_1,a_2 \in N(r) \setminus N(u)$ that will be contained in $X_3$;
at most two vertices $b_1,b_2 \in N(r) \setminus N(u)$ that will be contained in $X_4$;
an edge $a_3s_3$ such that $a_3 \notin N(u)$ and will be contained in $X_3$ and $s_3 \notin N(r)$ and will be contained in $X_5$;
an edge $b_3s_4$ such that $b_3 \notin N(u)$ and will be contained in $X_4$ and $s_4 \notin N(r)$ and will be contained in $X_5$;
and an edge $a_4b_4$ such that $a_4 \notin N(u)$ and will be contained in $X_3$ and $b_4 \notin N(r)$ and will be contained in $X_4$.
The sets $\{a_1,\ldots,a_4\}$, $\{b_1,\ldots,b_4\}$, $\{s_1,\ldots,s_4\}$ must be pairwise disjoint, but we not require that the vertices $a_1,\ldots,a_4$ are pairwise distinct, likewise we allow duplications in the vertices $b_1,\ldots,b_4$ and the vertices $s_1,\ldots,s_4$.

We now consider that graph $G'$ obtained from $G$ after removing all neighbours of $u$ and $r$ not guessed above. We do this as no neighbours of $u$ can be placed in $X_3$ and $X_4$, and no neighbours of $r$ can be placed in $X_5$. Further, from (iv) and (v), $N(u) \cap X_5 = \{s_1,s_2\}$ and $N(r) \cap (X_3 \cup X_4) = \{a_1,a_2,b_1,b_2\}$. Let $Z_1=\{a_1,\ldots,a_4\}$, $Z_2=\{b_1,\ldots,b_4\}$ and $Z_3=\{s_1,s_2\}$. As $(G',Z_1,Z_2,Z_3)$ is a $4$-bounded instance of {\sc $3$-Disjoint Connected Subgraphs}, by Theorem~\ref{t-kdisjoint}, we decide in polynomial time whether $V(G')$ contains three pairwise disjoint connected sets $S_1, S_2, S_3$ with $Z_i\subseteq S_i$ for every $i\in \{1,2,3\}$. If not, then we discard our guess as there is no $F_2$-model in $G'$ such that those guessed vertices are contained in the corresponding sets. Otherwise, we set $X_1=\{u\}$, $X_2=\{r\}$, $X_3=S_1$, $X_4=S_2$ and $X_5=S_3$. By construction, ${\cal X}$ is an $F_2$-model of $G$.

The correctness of our algorithm follows from its description. Since we apply Theorem~\ref{t-kdisjoint} $O(n^{12})$ times, our algorithm runs in polynomial time. 
\end{proof}

\section{The Graph $\mathbb{H}_2$}\label{s-h2}

We start with a structural lemma that follows immediately from Lemma~\ref{l:general}.

\begin{lemma}\label{l:bag-min-h}
For any bag-minimal $\mathbb{H}_2$-model $\mathcal X$ in a graph $G$, the following holds: 
\begin{itemize}    
    \item [(i)] $|X_{1}|=|X_{3}| = |X_{5}| = |X_{7}| = 1$.   
    \item [(ii)] $G_4$ is a path from unique $X_2$-attachment~$x$ in $X_4$ to a unique $X_6$-attachment~$y$ in $X_4$.
\end{itemize}
\end{lemma}

\noindent
\noindent
An $\mathbb{H}_2$-model~{$\cal X$} in a graph $G$ is \emph{small} if $|X_2|\leq 2$ or $|X_6|\leq 2$. Lemma~\ref{l:bag-min-h} has the following consequence:

\begin{corollary}\label{c-x4}
If a graph $G$ has a small $\mathbb{H}_2$-model, then $G$ has a small $\mathbb{H}_2$-model ${\cal X}$ with $|X_1|=|X_3|=|X_4|=|X_5|=|X_7|=1$.
\end{corollary}

\begin{proof}
Let $G$ be a graph that has a small $\mathbb{H}_2$-model ${\cal X}$. 
We may assume that ${\cal X}$ is bag-minimal, as deleting a vertex from a bag results in a new small $\mathbb{H}_2$-model in $G$.
From Lemma~\ref{l:bag-min-h}-(i), it now follows that $|X_1|=|X_3|=|X_5|=|X_7|=1$. Suppose $|X_4|\geq 2$. By Lemma~\ref{l:bag-min-h}-(ii), we find that $G_4$ is a path from a unique $X_2$-attachment~$x$ in $X_4$ to a unique $X_6$-attachment~$y$ in $X_4$. Since $x$ is the unique $X_2$-attachment in $X_4$, we can move all vertices of $X_4\setminus \{x\}$ to $X_6$ to obtain a small $\mathbb{H}_2$-model ${\cal X}'$ in $G$ with $|X_1'|=|X_3'|=|X_4'|=|X_5'|=|X_7'|=1$.
\end{proof}

\noindent
We use Corollary~\ref{c-x4} in the proof of the following lemma:

\begin{lemma}\label{lem-small-model}
It is possible to check in $O(n^9)$ time if a graph $G$ has a small $\mathbb{H}_2$-model.
\end{lemma}

\begin{proof}
Let $G$ be a graph. If $G$ has a small $\mathbb{H}_2$-model, then $G$ has a small
$\mathbb{H}_2$-model ${\cal X}$ with $|X_1|=|X_3|=|X_4|=|X_5|=|X_7|=1$ by Corollary~\ref{c-x4}. As ${\cal X}$ is small, by symmetry, we may assume that $|X_2|\leq 2$.
We therefore consider all $O(n^5)$ options of choosing vertices $x_1, x_3, x_4, x_5, x_7$ to form $X_1, X_3, X_4, X_5, X_7$, respectively. Afterwards, we consider all $O(n^2)$ options of choosing vertices $a_1$ and $a_2$ such that either $a_1=a_2$ or $a_1$ and $a_2$ are adjacent. We then let $X_2 = \{a_1,a_2\}$. 

We check in $O(n^2)$ time if
for all $i,j\in \{1,2,3,4,5,7\}$ with $i\neq j$, $X_i$ and $X_j$ are adjacent if and only if $ij\in E(\mathbb{H}_2)$. If not, we discard our guess. 
It remains to decide if there exists some $X_6 \subseteq V(G)$ such that $\cal X$ is an $\mathbb{H}_2$-model in $G$ with bags $X_1, \ldots, X_7$. Note that $X_6$ contains no neighbours of $x_1,x_3,a_1,a_2$. It follows that, if $\cal X$ contains some neighbour of $x_1,x_3,a_1,a_2$, then this is a guessed vertex.
Hence, we delete any non-guessed vertices that are neighbours of $x_1,x_3,a_1,a_2$. This yields a graph $G'$. We now check in $O(n^2)$ time if $G'$ contains a connected component $D$ that contains a neighbour of $x_4$, a neighbour of $x_5$ and a neighbour of $x_7$. If so, we set $X_2=V(D)$, which gives us a small $\mathbb{H}_2$-model ${\cal X}$ in $G'$ (and also $G$); otherwise we discard the guess as there is no $X_6$ such that $\cal X$ is an $\mathbb{H}_2$-model in $G$ with bags $X_1, \ldots, X_7$.
\end{proof}

\noindent
We now consider graphs $G$ with no small $\mathbb{H}_2$-model.  Let $P$ and $Q$ be two {\it mutually induced} paths on at least five vertices in $G$, that is, $V(P)\cap V(Q)=\emptyset$, and there is no edge with one end in $V(P)$ and the other in $V(Q)$. 
Let $x_1$ and $x_3$ be the ends of $P$. On $P$, let $a_1$ be the neighbour of $x_1$, and $a_2$ be the neighbour of $x_3$. Similarly, let $x_5$ and $x_7$ be the ends of $Q$, and on $Q$, let $b_1$ be the neighbour of $x_5$, and $b_2$ be the neighbour of $x_7$. We call the {\it $8$-tuple $(x_1, a_1, a_2, x_3, x_5, b_1, b_2, x_7)$} the {\it frame} of the pair $P,Q$. Note that since $P$ and $Q$ each have five vertices, the vertices $x_1, a_1, a_2, x_3, x_5, b_1, b_2, x_7$ are pairwise distinct. 

We say that a vertex $c\notin V(P)\cup V(Q)$ is a \emph{centre} for $P,Q$ if $c$ is complete to $(V(P)\cup V(Q))\setminus \{x_1,x_3,x_5,x_7\}$ and anti-complete  to
$\{x_1,x_3,x_5,x_7\}$. 
The reason for introducing these notions is the following: 
if $G$ has two mutually induced paths $P,Q$ on at least five vertices with centre~$c$, then $G$ contains
 an $\mathbb{H}_2$-model ${\cal X}$ with 
 $X_1=\{x_1\}$, $X_2=V(P)\setminus \{x_1,x_3\}$, 
 $X_3=\{x_3\}$, $X_4=\{c\}$, $X_5=\{x_5\}$, $X_6=V(Q)\setminus \{x_5,x_7\}$ and $X_7=\{x_7\}$.

We say that a pair of paths $P$ an $Q$ is {\it suitable} if $P$ and $Q$ are mutually induced paths on at least five vertices that have a center.
A suitable pair of paths $P, Q$ is {\it minimum} if there exist no suitable pair $P',Q'$ of $G$ with the same frame as $P,Q$, but for which 
 $|V(P')|+|V(Q')|<|V(P)|+|V(Q)|$.
 
For a frame ${\cal F}=(x_1, a_1, a_2, x_3, x_5, b_1, b_2, x_7)$, we let $G_{\cal F}$ be the {\it ${\cal F}$-reduced} graph obtained from $G$ after deleting
$x_1$, $x_3$, $x_5$, $x_7$, $b_1$ and $b_2$, together with all the neighbours of these vertices except for $a_1$ and $a_2$. So, $G_{\cal F} =G- (N[x_1, x_3, x_5, b_1, b_2, x_7] \setminus \{a_1, a_2\})$.

\begin{lemma}\label{lem-replacement}
    Let $G$ be a graph with no small $\mathbb{H}_2$-model that has a minimum suitable pair $P,Q$ with frame ${\cal F}=(x_1, a_1, a_2, x_3, x_5, b_1, b_2, x_7)$.
    Let $R$ be a shortest path from $a_1$ to $a_2$ 
    in $G_{\cal F}$. Let $P' = x_1 a_1 R a_2 x_3$. 
    It holds that $P',Q$ also form a minimum suitable pair of~$G$.
\end{lemma}

\begin{proof}
    We use a proof by contradiction. Suppose that the paths $P'$ and $Q$ are not mutually induced. That is, there exists some vertex $x$ in $P'$ such that either $x\in V(Q)$ or $x$ has some neighbour in $Q$. We choose 
    that vertex $x$ with minimum distance to $a_1$ on $P'$. Observe that $x\notin V(Q)$, else the neighbour of $x$ in $xP'a_1$ would be closer to $a_1$ and still has a neighbour in $Q$, a contradiction. Hence, $x$ has some neighbour~$q$ in $Q$.
    It follows from the definition of $G_{\cal F}$ that $x$ is not a neighbour of either $x_5$ or $x_7$.
     Hence, $q$ is an inner vertex of $Q$.

    By assumption, $x$ is not adjacent to $b_1$ or $b_2$. Therefore, there exists a subpath $\hat{Q}$ of $b_1Qb_2$ such that $x$ is complete to the set of internal vertices of $\hat{Q}$ and anti-complete to the ends of $\hat{Q}$.
 By construction, $x$ is not adjacent to either $x_1$ or $x_3$.
 If $x$ has some neighbour in $P$, then there also exists a subpath $\hat{P}$ of $x_1Px_3$ such that $x$ is complete to the set of  internal vertices of $\hat{P}$ and anti-complete to its ends of $\hat{P}$. 
 If both $\hat{P}$ and $\hat{Q}$ have at least five vertices, this means that $\hat{P}, \hat{Q}$ is a suitable pair of $G$ with 
 $|V(\hat{P})|+|V(\hat{Q})|<|V(P)|+|V(Q)|$, a contradiction.
 If one of $\hat{P}, \hat{Q}$ has at most four vertices, we find that 
 $G$ has a small $\mathbb{H}_2$-model, another contradiction.
 It follows that $x$ has no neighbour in $P$.

Since $a_1\in V(P)$, there exists a vertex $y$ on the subpath $a_1Rx$, such that $y$ has some neighbour in $a_1Pa_2$. We choose such a vertex $y$ with minimum distance to $x$ on $R$. Note that since $a_1 \in V(P)$, such a $y$ must exist, and in particular $y \notin \{x,a_1\}$. 
    
First suppose that $y$ is a adjacent to two non-adjacent vertices $p_1$ and $p_2$ on $a_1Pa_2$. We set $X_1=\{p_1\}$, $X_2=\{y\}$, $X_3=\{p_2\}$,
$X_4 = V(yRx)\setminus \{y\}$, $X_5=\{x_5\}$, $X_6=V(Q)\setminus \{x_5,x_7\}$ (note $q\in X_6$)
and $X_7=\{x_7\}$ to obtain a small $\mathbb{H}_2$-model in $G$, a contradiction.
Hence, $y$ is adjacent to at most two vertices $d_1,d_2$ in $a_1Pa_2$. Further, if $d_1 \neq d_2$, then $d_1$ and $d_2$ are adjacent.
Suppose that the order of vertices on $P$ is $x_1, a_1, d_1, d_2, a_2, x_2$. Let $d_1'$ be the neighbour of $d_1$ closest to $x_1$, and let $d_2'$ be the neighbour of $d_2$ closest to $x_2$.
We set $X_1=\{d_1'\}$, $X_2=\{d_1,d_2\}$, $X_3=\{d_2'\}$,
$X_4 = V(yRx)$, $X_5=\{x_5\}$, $X_6=V(Q)\setminus \{x_5,x_7\}$ 
(note $q\in X_6$) and $X_7=\{x_7\}$ to obtain a small $\mathbb{H}_2$-model in $G$. Again, this is a contradiction. We conclude that $P'$ and $Q$ are mutually induced. 

Since $P,Q$ form a suitable pair, there exists a vertex $c$ that is a centre for $P,Q$. We claim that $c$ is also a centre for $P',Q$.
For a contradiction, suppose $c$ is not a centre for $P',Q$.
Because $c$ is the centre for $P,Q$, we find that $c$ is anti-complete to $\{x_1,x_3,x_5,x_7\}$ and complete to $\{a_1,a_2\}\cup (V(Q)\setminus \{x_3,x_5\})$. 
Since $c$ is not a centre for $P',Q$, this means that $c$ is not adjacent to some vertex $z\in V(R)$. We choose $z$ such that its distance to $a_1$ along $R$ is minimised. 
As $a_1\in V(x_1P'z)$, we find that $c$ has a neighbour in $V(x_1P'z)$.
Moreover, if $x_1P'z$ has at most four vertices, then $V(x_1P'z) \setminus \{x_1,z\}$ has size at most~$2$. 
In that case, setting $X_1=\{x_1\}$, $X_2= V(x_1P'z) \setminus \{x_1,z\}$, $X_3=\{z\}$,
$X_4 = \{c\}$, $X_5=\{x_5\}$, $X_6=V(Q)\setminus \{x_5,x_7\}$, $X_7=\{x_7\}$ yields a small $\mathbb{H}_2$-model in $G$, a contradiction. 
Hence, $x_1P'z$ contains at least five vertices, just like $Q$.
Moreover, $c$ is complete to the internal vertices of both $x_1P'z$ and $Q$, and $c$ is anti-complete to $\{x_1, z, x_5, x_7\}$. Therefore, $x_1P'z$ and $Q$ form a suitable pair. 

We note that $a_1Pa_2 \subseteq G_{\cal F}$.
Since $R$ was a shortest path from $a_1$ to $a_2$ in $G_{\cal F}$, it holds that $|V(R)|\leq |V(a_1Pa_2)|$ and thus 
$|V(P')| \leq |V(P)|$. As $z\in V(R)$, $z \neq x_3$, and thus $x_1P'z \subsetneq P'$. This means that $|V(x_1P'z)| < |V(P)|$. 
Hence, $x_1P'z, Q$ form a suitable pair of~$G$ that has the same frame as $P,Q$ but with $|V(x_1P'z)|+|V(Q)| < |V(P)|+|V(Q)|$, contradicting our assumption that $P,Q$ is minimum. It follows that $c$ is a centre for $P',Q$, and thus, $P',Q$ form a suitable pair of $G$, which is even minimum
as $|V(P')| \leq |V(P)|$.
\end{proof}

\begin{figure}
    \centering
    \includegraphics[width=0.35\linewidth,page=3]{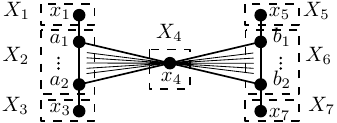}
    \caption{The $\mathbb{H}_2$-model ${\cal X}$ in $G$ from Lemma~\ref{lem-type-1or2}.}
    \label{fig:types}
\end{figure}

\noindent
We prove the following lemma.

\begin{lemma}\label{lem-type-1or2}
    Let $G$ be a graph with no small $\mathbb{H}_2$-model. For every minimum $\mathbb{H}_2$-model $\mathcal X$ in $G$ that maximises $|X_4|$, 
    the following four statements hold 
    (see also Figure~\ref{fig:types}):
\begin{itemize}    
    \item [(i)] $X_1=\{x_1\}$, $X_3=\{x_3\}$, $X_4=\{x_4\}$, $X_5=\{x_5\}$, $X_7=\{x_7\}$ for $x_1, x_3, x_4,x_5, x_7\in V(G)$;
    \item [(ii)] $G_2$ is a path from a unique $X_1$-attachment $a_1$ to a unique $X_3$-attachment $a_2$;
    \item [(iii)] $G_6$ is a path from a unique $X_5$-attachment $b_1$ to a unique $X_7$-attachment $b_2$; 
    \item [iv)]  $x_4$ is complete to $X_2 \cup X_6$. 
  \end{itemize}
\end{lemma}

\begin{proof}
Let ${\cal X}$ be a minimum $\mathbb{H}_2$-model $\mathcal X$ in $G$ that maximises $|X_4|$. As $\mathcal X$ is minimum, ${\cal X}$ is bag-minimal.
By Lemma~\ref{l:bag-min-h}-(i), we have $X_1=\{x_1\}$, $X_3=\{x_3\}$, $X_5=\{x_5\}$ and $X_7=\{x_7\}$ for some $x_1, x_3, x_5, x_7\in V(G)$.
By Lemma~\ref{l:bag-min-h}-(ii), $G_4$ is a path from a vertex $x$ to a vertex $y$.

Let $R$ be a shortest path in $G_2$ between some $X_1$-attachment vertex~$a_1$ and some $X_3$-attachment~$a_2$. 
 We claim that $x$ has a neighbour in~$R$. Otherwise, let $Z$ be a shortest path in $G[X_2\cup \{x\}]$ from $x$ to some vertex with a neighbour in~$R$. Replacing $X_2$ by~$R$ and $X_4$ by~$Z \cup X_4$ yields a minimum $\mathbb{H}_2$-model~${\cal X}'$ in~$G$ with~$|X_4'|>|X_4|$, contradicting our choice of~${\cal X}$. Moreover, $G_2 = R$, else replacing $X_2$ by $V(R)$ would yield a smaller $\mathbb{H}_2$-model in~$G$, another contradiction. The claim about $X_6$ holds similarly. Hence, we have proven (ii) and (iii).

Now suppose $|X_4|\geq 2$, so in particular $x\neq y$. If $x$ is not adjacent to $a_1$, then we replace $X_1$ by $\{a_1\}$ to find a smaller $\mathbb{H}_2$-model in $G$. Hence, $x$ is adjacent to $a_1$. For the same reason, $x$ is also adjacent to $a_2$, and $y$ is adjacent to $b_1$ and $b_2$. If $|X_2| \geq 3$, then we obtain a smaller $\mathbb{H}_2$-model in $G$ by deleting $x$ from $X_4$ and letting $X_2 = \{x\}$, $X_1 = \{a_1\}$ and $X_3 = \{a_2\}$, a contradiction. Hence, $X_2 = \{a_1, a_2\}$, in which case $G$ has a small $\mathbb{H}_2$-model, a contradiction.  Hence, $X_4= \{x\}$, which completes the proof of (i).

It remains to show (iv).
Suppose $x$ has a non-neighbour $z$ in~$X_2$. We choose $z$ such that its distance to $a_1$ along $X_4$ is minimised. Replacing $X_2$ by $V(a_1 X_2 z) \setminus \{z\}$ and $X_3$ by $\{z\}$ yields a smaller $\mathbb{H}_2$-model in $G$, a contradiction.  Hence, $x$ is complete to $X_2$, and by symmetry, to $X_6$. This proves~(iv).
\end{proof}

\noindent
 Recall that for a frame ${\cal F}=(x_1, a_1, a_2, x_3, x_5, b_1, b_2, x_7)$, the ${\cal F}$-reduced graph is the graph $G_{\cal F} =G- (N[x_1, x_3, x_5, b_1, b_2, x_7] \setminus \{a_1, a_2\})$.
 We now also define the {\it reversed} ${\cal F}$-reduced graph of ${\cal F}$ as 
 $G_{\cal F}^r =G- (N[x_1, x_3, x_5, a_1, a_2, x_7] \setminus \{b_1, b_2\})$.

We are now ready to present our algorithm. In its correctness proof, we use the shortest path detector technique of \cite{chudnovsky.c.l.s.v:reco}, namely in two consecutive applications of Lemma~\ref{lem-replacement}.

\begin{theorem}
\label{th:detect}
$\mathbb{H}_2$-{\sc Induced Minor} is polynomial-time solvable.
\end{theorem}
\begin{proof}
We describe an $O(n^9)$-time algorithm.
First apply the $O(n^9)$ time algorithm of Lemma~\ref{lem-small-model} to decide whether the input $G$ contains a small $\mathbb{H}_2$-model~${\cal X}$.
If so, output ${\cal X}$. Else, consider all $O(n^8)$ $8$-tuples $(x_1, a_1, a_2, x_3, x_5, b_1, b_2, x_7)$ that induce a subgraph of $G$ with exactly the four edges $x_1a_1$, $a_2x_3$, $x_5b_1$, $b_2x_7$. For each $8$-tuple, do as follows. Compute a shortest path~$R$ from $a_1$ to $a_2$ in $G_{\cal F}$ and a shortest path $R'$ from $b_1$ to $b_2$ in $G_{\cal F}^r$. Let $P = x_1 a_1 R a_2 x_3$ and $Q = x_5 b_1 R' b_2 x_7$. Check if $P$ and $Q$ are mutually induced and if there exists a vertex $c$ in $G- N[x_1, x_3, x_5, x_7]$ that has a neighbour in $P$ as well as a neighbour in $Q$. If so, output the $\mathbb{H}_2$-model~${\cal X}$ in $G$ with $X_1 = \{x_1\}$, $X_3 = \{x_3\}$, $X_5 = \{x_5\}$, $X_7 = \{x_7\}$, $X_4= \{c\}$, $X_2= V(R)$ and $X_7= V(R')$. After considering all tuples, output no and stop.

As processing each $8$-tuples takes $O(n)$ time and there are $O(n^8)$, our algorithm runs in time $O(n^9)$. We now prove correctness. It is readily seen that any output ${\cal X}$ is indeed an $\mathbb{H}_2$-model of $G$. Now suppose $G$ contains $\mathbb{H}_2$ as an induced minor. We will show that our algorithm will find an $\mathbb{H}_2$-model in $G$.

If $G$ contains $\mathbb{H}_2$ as an induced minor, then there is a minimum $\mathbb{H}_2$-model $\mathcal{X}$ in $G$ that maximises $X_4$. As $G$ contains no small $\mathbb{H}_2$-model, we can apply Lemma~\ref{lem-type-1or2}. It follows that

    \begin{itemize}
        \item [(i)] $X_1=\{x_1\}$, $X_3=\{x_3\}$, $X_4=\{x_4\}$, $X_5=\{x_5\}$, $X_7=\{x_7\}$ for $x_1, x_3, x_4,x_5, x_7\in V(G)$;
    \item [(ii)] $G_2$ is a path from a unique $X_1$-attachment $a_1$ to a unique $X_3$-attachment $a_2$;
    \item [(iii)] $G_6$ is a path from a unique $X_5$-attachment $b_1$ to a unique $X_7$-attachment $b_2$; 
    \item [iv)]  $x_4$ is complete to $X_2 \cup X_6$ (but anti-complete to $\{x_1,x_3,x_5,x_7\}$). 
\end{itemize} 

\noindent
Hence, $P = x_1a_1G_2a_2x_3$ and $Q = x_5b_1G_6b_2x_7$ are mutually induced paths with frame $(x_1, a_1, a_2, x_3, x_5, b_1, b_2, x_7)$ and centre $x$. That is, $P, Q$ form a suitable pair, which is minimum, as
$\mathcal{X}$ is minimum.
When our algorithm considers the $8$-tuple $(x_1, a_1, a_2, x_3, x_5, b_1, b_2, x_7)$, it computes a shortest path~$R$ from $a_1$ to $a_2$ in $G_{\cal F}$ and a shortest path $R'$ from $b_1$ to $b_2$ in $G_{\cal F}^r$. Let $P = x_1 a_1 R a_2 x_3$ and $Q = x_5 b_1 R' b_2 x_7$.
As $P, Q$ form a minimum suitable pair of paths, the same holds for $P',Q$ by Lemma~\ref{lem-replacement}.
As $P', Q$ form a minimum suitable pair of paths, the same holds for $P',Q'$, again by Lemma~\ref{lem-replacement}.
As $P',Q'$ form a suitable pair, 
the pair $P', Q'$ has some centre $x_4$.
This means that $x_4$ is in $G-N[x_1,x_3,x_5,x_7]$ and has a neighbour in $R$ and a neighbour in $R'$. Hence, our algorithm will find $x_4$ or a vertex with the same properties as~$x_4$, and outputs an $\mathbb{H}_2$-model of $G$. We conclude that our algorithm is correct.
\end{proof}

\section{Windmills}\label{s-np}

For positive integers $a$, $b$, $c$ and $d$, a graph $G$ is an \emph{$(a, b, c, d)$-windmill} if $G$ contains two mutually induced paths $P= u_1  \dots u_k$ on  $k\geq a+b+2$ vertices and $Q= v_1  \dots v_\ell$ on $\ell\geq c+d+2$ vertices, plus a vertex $z$ called the \emph{centre} that is anti-complete to $X = \{u_1, \dots, u_a\} \cup \{u_{k-b+1}, \dots, u_k\} \cup \{v_1, \dots, v_c\} \cup \{v_{\ell-d+1}, \dots, v_\ell\}$ and complete to $(V(P)\cup V(Q)) \setminus X$. It follows that $z$ has degree at least $4$. 
For fixed $a,b,c,d$, the {\sc $(a, b, c, d)$-Windmill} problem is to decide if a given graph contains an $(a, b, c, d)$-windmill. 

Detecting a windmill is a problem similar to the main subroutine in the proof of Theorem~\ref{th:detect}. However, some of its variants are in general \NP-complete. We rely on the following \NP-completeness results.  By {\sc 2-in-a-hole} we mean the problem whose instance is a graph~$G$ with two prescribed vertices $x$ and $y$ of degree~2, and whose question is whether there is some chordless cycle of $G$ containing $x$ and $y$. By {\sc Induced 2-Disjoint Paths} we mean the problem whose instance is a graph $G$ with four prescribed vertices $x'$, $x''$, $y'$, $y''$ and whose question is whether $G$ contains two mutually disjoint paths from $x'$ to $y'$ and $x''$ to $y''$. 
A \emph{hub} in a graph is a vertex that has at least three neighbours of degree at least~3. 

\begin{theorem}[see \cite{diotTaTr:13}]
\label{th:diot}
 The problems  {\sc 2-in-a-hole} and {\sc Induced 2-Disjoint Paths} are \NP-complete, even when restricted to hub-free instances. 
\end{theorem}

\noindent
Observe that in \cite{diotTaTr:13}, only {\sc $2$-in-a-Hole} is proven to be NP-complete, but the proof that is given also applies (with no modification) to {\sc Induced $2$-Disjoint Paths}.  Also note that solving {\sc $2$-in-a-Hole} can be performed by solving two instances of {\sc Induced $2$-Disjoint Paths}, so it is an easier problem.
  
\begin{theorem}\label{t-np}
    If $a$, $b$, $c$ and $d$ are four positive integers, no three of which are equal, then {\sc $(a, b, c, d)$-Windmill} is \NP-complete. 
\end{theorem}

\begin{proof}
    We first address the case when $\{a, b\} = \{c, d\}$, say $a=c$, $b=d$ up to symmetry (so $a\neq b$ since no three integers among $a$, $b$, $c$ and $d$ can be equal). We consider an instance  $(H, x, y)$ of {\sc  2-in-a-hole} and denote by $x'$, $x''$ and $y'$, $y''$ the neighbours of $x$ and $y$ respectively.  
    
    Let us build a graph $G$ from $H$ by removing $x$ and $y$, adding a vertex $z$ adjacent to all vertices known so far, four paths $x'_1\dots x'_a$,  $y'_1\dots y'_b$, $x''_1\dots x''_c$ and  $y''_1\dots y''_d$ and the four edges $x'x'_a$, $y'y'_b$, $x''x''_c$ and $y''y''_d$.  We claim that $H$ contains a chordless cycle containing $x$ and $y$ if and only if $G$ contains an $(a, b, c, d)$-windmill. 

    If $H$ contains a chordless cycle containing  $x$ and $y$, then this cycle must contain either two mutually induced paths from $x'$ to $y'$ and from $x''$ to $y''$, or two mutually induced paths from $x'$ to $y''$ and from $x''$ to $y'$.  In either case, these paths yield an $(a, b, c, d)$-windmill in~$G$.  Conversely, if some $(a, b, c, d)$-windmill exists in~$G$, it must be centred at $z$ since $H$ is hub-free.  Hence, it must contain the four paths $x'_1\dots x'_a$,  $y'_1\dots y'_b$, $x''_1\dots x''_c$ and  $y''_1\dots y''_d$ and either two mutually induced paths from $x'$ to $y'$ and from $x''$ to $y''$, or two mutually induced paths from $x'$ to $y''$ and from $x''$ to $y'$. Note that mutually induced paths from $x'$ to $x''$ and $y'$ to $y''$ cannot manifest as a solution since $a\neq b$.  In either case, these two paths yield a chordless cycle containing $x$ and $y$ in $H$.

    We now address the case where $a = b$ and $c =d$ (so $a\neq c$).  The reduction to the detection of a windmill is entirely similar, except that now we prove that $G$ contains an $(a, b, c, d)$-windmill if and only if $H$ contains two mutually disjoint paths from $x'$ to $y'$ and $x''$ to $y''$ (so we rely on the \NP-completeness of {\sc Induced 2-Disjoint Paths}). 

    The remaining cases follow in the same fashion (from the \NP-completeness of {\sc Induced 2-Disjoint Paths}).
\end{proof}

\noindent
One may observe that according to our proof, some other (possibly meaningful) variants of detecting a windmill can also be shown to be \NP-complete, but we are not sure there is sufficient motivation to list all such possible problems (see also Section~\ref{s-con} for some further discussion).

\section{Conclusions}\label{s-con}

By resolving the three open cases of graphs on five vertices and 
one long-standing remaining case of a tree on seven vertices, we settled the complexity of {\sc $H$-Induced Minor} for all graphs $H$ on five vertices and for all forests $H$ on at most seven vertices. We finish our paper with some natural open problems. 

An interesting case, posed as an open problem by Fiala et al.~\cite{FKP12},
is where $H$ is the double star in which both centre vertices have exactly three leaves. We note that the polynomial-time algorithm in~\cite{FKP12} for the case where $H$ is a double star in which one of the centre vertices has at most two leaves cannot be used.
Another interesting open case on six vertices, which follows naturally from the result of Bousquet et al.~\cite{BDDHMPT26} for $K_{2,3}$-{\sc Induced Minor}, is the graph $H=K_{3,3}$: the complete bipartite graph with three vertices on both sides. 

We note that our reduction in Theorem~\ref{t-np} does not produce any \NP-completeness result when at least three numbers among $a$, $b$, $c$ and $d$ are equal.
For instance, the case $a=b=c=d=1$ is still open. The problem is that in this case, mutually induced  paths from $x'$ to $x''$ and  from $y'$ to $y''$ in our reduction yield a windmill.  So, we have no ``anchor''  to force the windmill to go from the $x$-zone to the $y$-zone of $H$.

We give one more related open problem: what is the complexity of  {\sc $4$-in-$2$-Paths}? This problem was introduced in~\cite{Pa11}
and has as input a graph $G$ with four prescribed vertices $x$, $x'$, $y'$, $y''$.  The question is whether $G$ contains two mutually disjoint paths $P$ and $Q$, such that all ends of $P$ and $Q$ are in $\{x', x'', y', y''\}$. If {\sc $4$-in-$2$-Path} is \NP-complete for  hub-free graphs, then {\sc $(a, b, c, d$)-Windmill} would be \NP-complete for all $a$, $b$, $c$, $d$.


\bibliographystyle{plainurl}
\bibliography{inducedminor}
\end{document}